\documentclass[a4paper,12pt,titlepage]{article}
\usepackage{graphicx}
\usepackage{amsmath}
\usepackage{float}
\usepackage{amssymb}

\newtheorem{theorem}{Theorem}

\newtheorem{definition}[theorem]{Definition}

\newtheorem{lemma}[theorem]{Lemma}

\newtheorem{proposition}[theorem]{Proposition}

\newenvironment{proof}[1][Proof]{\textbf{#1.} }{\ \rule{0.5em}{0.5em}}
\begin{document}

\title{Periodic Patrols on the Line and Other Networks}
\author{Steve Alpern$^{1}$, Thomas Lidbetter$^{2}$ and Katerina Papadaki$%
^{3} $ \\
$^{1}$ORMS Group, Warwick Business School, \\University of Warwick, Coventry
CV4, UK\\
$^{2}$Department of Management Science and Information Systems, \\Rutgers Business School, NJ 07102, USA\\
$^{3}$Department of Mathematics, London School of Economics, \\London WC2A 2AE, UK}
\maketitle

\begin{abstract}
We consider a patrolling game on a graph recently introduced by Alpern et al. (2011) where the Patroller wins if he is at the attacked node while the attack is taking place. This paper studies the periodic patrolling game in the case that the attack duration is two periods. We show that if the Patroller's period is even, the game can be solved on any graph by finding the {\em fractional covering number} and {\em fractional independence number} of the graph. We also give a complete solution to the periodic patrolling game on
line graphs of
arbitrary size, extending the work of Papadaki et al. (2016) to the periodic domain. This models the patrolling problem on a border or channel, which is related to a classical problem of operational research going back to Morse and Kimball
(1951). A periodic patrol is required to start and end at the same location,
for example the place where the Patroller leaves his car to begin a foot patrol.
\end{abstract}

\section{Introduction}

The periodic patrolling game was introduced in Alpern et al. (2011)
to model the defense
of the nodes of a network from attack by an
antagonistic opponent. This is a discrete game model in which the network is
modeled as a graph, the Patroller chooses a walk on the graph with a given
period and the Attacker picks a node and a discrete time interval of fixed
duration $m$ for his attack. The Patroller wins the game if he is present at
the attacked node during the time interval in which it is attacked, in which
case we say that he intercepts the attack. Otherwise the Attacker wins.
Compared with other patrolling models in the literature, for example Chung
et al. (2011), the patrolling game model represents only
an idealization of the patrolling problem. However it is the only model in
which the Patroller and Attacker are treated symmetrically, rather than the
more usual Stackelberg approach where the Patroller picks his strategy first.

This paper considers the periodic
patrolling game on general graphs and then in more detail on the class of
line graphs $L_{n}$ consisting of $n$ nodes $1,2,\dots ,n$ with consecutive
numbers considered to be adjacent. The case of a unit attack duration $m=1$
is covered by the field of geometric games as defined by Ruckle (1983), so
we here consider the next smallest duration $m=2,$ which is the only case
thus far susceptible to analysis. We note that the easier version of
non-periodic patrolling games is able to handle line graphs for larger
values of $m$, as recently solved by Papadaki et al. (2016). It is likely that
the techniques introduced here will be extended to larger attack durations
in the future, but clearly additional ideas will be required.

In the case of the line graph, our discrete model could be applied for example to the problem of
patrolling, possibly with a sniffer dog, a bank of linearly arranged airport
security scanners, or a mountainous border with a discrete set of passes
that can be crossed. In such cases, the ``nodes'' can be attacked at any time,
around the clock, so the period $T$ is likely to be the number of nodes that
can be patrolled in a day. Other possiblities for defining $T$ might be the
attention span of the sniffer dog or the time between refueling by a mobile vehicle, robot or UAV.

The paper is organized as follows. In Section~\ref{sec:lit}, we review the related literature, then in Section~\ref{sec:def} we formally define the game. In Section~\ref{sec:gen_graphs} we discuss some results for general graphs, showing how the game can be solved using notions from fractional graph theory if the patrol period is even. We then give a complete solution to the game played on a line graph in Section~\ref{sec:line}. In Section~\ref{sec:multiple} we consider an extension of the game to the case of multiple patrollers, and show how our results on the line may be extended to this setting. Finally, we conclude in Section~\ref{sec:conclusions}.

\section{Literature review}
\label{sec:lit}

As stated in the abstract, the problem of patrolling a border or channel
against attack or infiltration goes back to the classical work of Morse and
Kimball (1951). Since then many attempts have been made to improve the theory
and practice of patrolling. Washburn (1982) considers an infiltrator who wants
to maximize the probability of getting across a line in a channel. The case
where the channel is blocked by fixed barriers has been consider by Baston and
Bostock (1987) and the case when the barriers are moving has been analyzed by
Washburn (2010). The case of a thick infiltrator has been considered by Baston
and Kikuta (2009). If there are many infiltrators and they arrive in a Poisson
manner, the analysis is given by Szechtman et al. (2008). Multiple infiltrators
are also considered by Zoroa et al. (2012) where the infiltration is through a
circular rather than a linear boundary. Multiple patrollers, when only some
portions of the boundary need to be protected, are considered by Collins et al.
(2013), who show how the problem can be divided up. Papadaki et al. (2016)
consider the discrete border patrol problem, where the infiltration can only
be accomplished at certain points of the border (perhaps mountain passes).
When patrollers are restricted to periodic patrols, as here, the analysis of
the continuous problem (with elements such as turning radius included) has
been analyzed by Chung et al. (2011).

The more general problem of patrolling an arbitrary network against attacks at
its nodes has been modeled as a game by Alpern et al. (2011), including a
definition of the periodic patrolling game which we adopt here. Lin et al.
(2013) developed more general approximate methods which cover such extensions
as varying values for attacks at different nodes. Their methods, extended in
Lin et al. (2014) to imperfect detection, can solve large scale problems. In
the computer science literature, patrolling games with mobile robots and a
Stackelberg model have been developed by Basilico et al. (2009, 2012). Multi
vehicle patrolling problems have been solved by Hochbaum et al. (2014).

Infiltration games without mobile patrollers are analyzed in Garnaev et al.
(1997), Alpern (1992), Baston and Garnaev (1996) and Baston and Kikuta (2004, 2009).

\section{Formal Definition of the (periodic) Patrolling Game}
\label{sec:def}

In this section we formally define the patrolling game. There are three
parameters: a graph $Q=Q(N,E)$ (where $N$ is the set of nodes and $E$ is the set of edges of $Q$), a period $T$, and an attack
duration $m$ (which we will take as $2$ in this paper). The Attacker
chooses a node $i$ of $Q$ to attack and a time interval of $m$
consecutive periods in which to attack it. These $m$ periods can be
considered as an arc of the time circle $\mathcal{T}=\left\{ 1,2,\dots
,T,T+1=1\right\}$, on which arithmetic is carried out modulo $T$. So in the
periodic game with $T=24$ and $m=5$, for example, a valid Attacker strategy
would be the \textquotedblleft overnight\textquotedblright\ attack, with attack interval
$J =\{22,23,24,1,2\}$. Note that if $Q$ has $n$ nodes, then the number of
possible attacks is given by $nT,$ and the mixed attack strategy which
chooses among them equiprobably will be called the \textit{uniform attack
strategy}. To foil the attack, the Patroller walks along the graph in an
attempt to intercept it, that is, to be at the attacked node at some time
during the attack interval. More precisely, a patrol is a walk $w$ on $Q$
with period $T$, that is, $w:\{1,2,\ldots\} \rightarrow N$ with $w(t)$ and $w(t+1)$ the same or adjacent nodes and $w\left( t+T\right) =w\left( T\right) $ for all $t$. A patrol $w$
intercepts an attack at node $i$ during attack interval $J$ if $i\in w\left(
J\right) $ or equivalently if $w\left( t\right) =i$ for some time $t$ in the
attack interval $J$. In such a case we say that the Patroller wins, and the
payoff is $1$; otherwise we say the Attacker wins, and the payoff is $0$.
Thus the payoff of the game corresponding to mixed strategies is the probability that the Patroller
intercepts the attack. The value $V$ of the game is the expected payoff (interception probability)
with optimal play on both sides.

We note that in Alpern et al. (2011), this game is called the periodic
patrolling game (one of two forms of the game considered there) and the
value is denoted $V^{p}$. We assume throughout that the period is at least
$2$ and that the graph $Q$ has at least $n=2$ nodes.

\section{General Graphs}
\label{sec:gen_graphs}

In this section we obtain some bounds on the value $V$ of the patrolling
game on a general graph. The tools comprise the well known covering and
independence numbers and a decomposition result taken from Alpern et al.
(2011).

\subsection{Covering and independence numbers $\mathcal{I}$ and $\mathcal{C}%
. $}
\label{sec:general_results}

We recall some elementary definitions about a graph $Q$. A set of nodes is
called {\em independent} if no two of them are adjacent. The maximum cardinality of an
independent set is called the independence number $\mathcal{I}$.
Similarly a set of edges is called a {\em covering set} if every node of the graph
is incident to one of these edges. The minimum cardinality of a covering set
is called the covering number $\mathcal{C}$ of the graph. It is well known that $\mathcal{I%
}\leq \mathcal{C}$.

Suppose the Attacker attacks in some fixed time interval $\{t,t+1\}$ at a
node chosen equiprobably from a set of $\mathcal{I}$ independent nodes. We call
this an {\em independent attack strategy}. If a
patrol intercepts one of these attacks at node $i\in \mathcal{I}$ at time $%
t, $ he cannot intercept another at time $t+1,$ since none of the other
attacks are at a node adjacent to $i$. Hence the probability of intercepting
an attack cannot exceed $1/\mathcal{I}$ and therefore $V\leq 1/\mathcal{I}$.
Next suppose $T$ is even. In this case the Patroller fixes a covering set
of $\mathcal{C}$ edges, picks a single edge amongst these randomly, and on
that edge goes back and forth in an oscillations of length $T$. We call this Patroller mixed strategy
an {\em unbiased covering strategy}, or, if the covering set is only an edge, an {\em unbiased oscillation}. Every node is visited by one of these
patrols in every pair of consecutive time periods, and hence every attack of
duration $m=2$ is intercepted by at least one of these $\mathcal{C}$ patrols.
Therefore the Patroller wins with probability at least $1/\mathcal{C}$. Hence we have
shown the following.

\begin{lemma}
The value of the Patrolling Game on any graph $Q$ satisfies%
\begin{align}
V &\leq 1/\mathcal{I}, \text {and futhermore}  \label{V<1/I} \\
1/\mathcal{C} &\leq V\leq 1/\mathcal{I},\text{ if $T$ is even.}
\label{1/C<V<1/I}
\end{align}
\label{lemma:cover-independence}
\end{lemma}

A graph is called bipartite if its nodes can be partitioned into two sets
such that no two nodes within the same set are adjacent. For bipartite
graphs, we can say more.

\begin{proposition}
\label{prop:cover-indep}Let $Q$ be a bipartite graph. Then $\mathcal{C}=\mathcal{I}$ and
the value $V$ satisfies%
\begin{eqnarray}
V &=&\frac{1}{\mathcal{C}}=\frac{1}{\mathcal{I}},\text{ if }T\text{ is even,
and}  \label{bipartite} \\
\left (\frac{2T-1}{2T}\right )\frac{1}{\mathcal{C}} = \left(\frac{2T-1}{2T} \right)\frac{1}{\mathcal{I}} &\leq &V\leq \frac{1}{\mathcal{I}}\text{ if }T\text{ is odd.%
}  \label{VToddT}
\end{eqnarray}
\end{proposition}

\begin{proof}
The first result (\ref{bipartite}) follows immediately from (\ref{1/C<V<1/I}%
) and the well known fact (Konig's Theorem) that $\mathcal{C}=\mathcal{I}$
for bipartite graphs. The upper bound of (\ref{VToddT}) follows from (\ref%
{V<1/I}). For the lower bound let $\{e_{k}\}_{k=1}^{\mathcal{C}}$ be a covering
set of $\mathcal{C}$ edges, and let $w_{k}$ denote the randomized walk of
period $T$ which oscillates on $e_{k}$ except that it stays at a randomly chosen node
of $e_{k}$ for two consecutive times, also randomly chosen. We call this strategy of the Patroller a {\em biased
covering strategy}. For example if $T=7$
and the endpoints of $e_{k}$ are $a$ and $b$, the repeated sequence might be $%
ababbab$. Consider the Patroller strategy that chooses one of the randomized walks $%
w_{k}$ equiprobably. If one of the nodes of $e_{k}$ is attacked then the
attack is detected if the Patroller chooses $w_{k}$ (which happens with
probabiliy $1/\mathcal{C}$) and the Patroller does not happen to choose to
repeat this node for two consecutive periods that coincide with the time of
attack (this happens with probability $1-1/(2T)$. So the total probabilty
the attack is detected is $\left( 1/\mathcal{C}\right) \left( 1-1/\left(
2T\right) \right) $, giving the lower bound for the value in (\ref{VToddT}).
\end{proof}

We now give an example based on Lemma \ref{lemma:cover-independence} and Proposition \ref{prop:cover-indep} for the line graph $L_7$ with nodes $\{1,\ldots,7\}$ and edges $(i,i+1)$ $i=1,\ldots,6$. Since $L_n$ is bipartite we can use the result in (\ref{bipartite}). We demonstrate the result for even period $T=12$ (any even period would suffice but we pick $12$ to be able to compare it with a later example in Section~\ref{sec:non-periodic}). A minimum covering set is $\{(1,2), (3,4), (5,6), (6,7)\}$ and thus $\mathcal{C} = 4$. An unbiased covering strategy for the Patroller consists of picking an edge at random from a minimum covering set (with probability $1/4$) and performing an oscillation on that edge with period $T=12$. Since $T=12$ is even the oscillations performed on the chosen edges are unbiased (nodes are visited equally often). This is demonstrated in Figure \ref{fig:example_periodic}. This Patroller strategy intercepts attacks at nodes $1-5,7$ with probability $1/4$ and at node $6$ with probability $1/2$. Thus, the Patroller at worst can guarantee interception probability of at least $1/4$. The Attacker would use the independent attack strategy and attack equiprobably on the independent set $\{1,3,5,7\}$, which clearly guarantees him interception probability of at most $1/4$. This gives the value of the game $V=1/C=1/4$.

\begin{figure}[H]
	\begin{center}
		\includegraphics[scale=0.3]{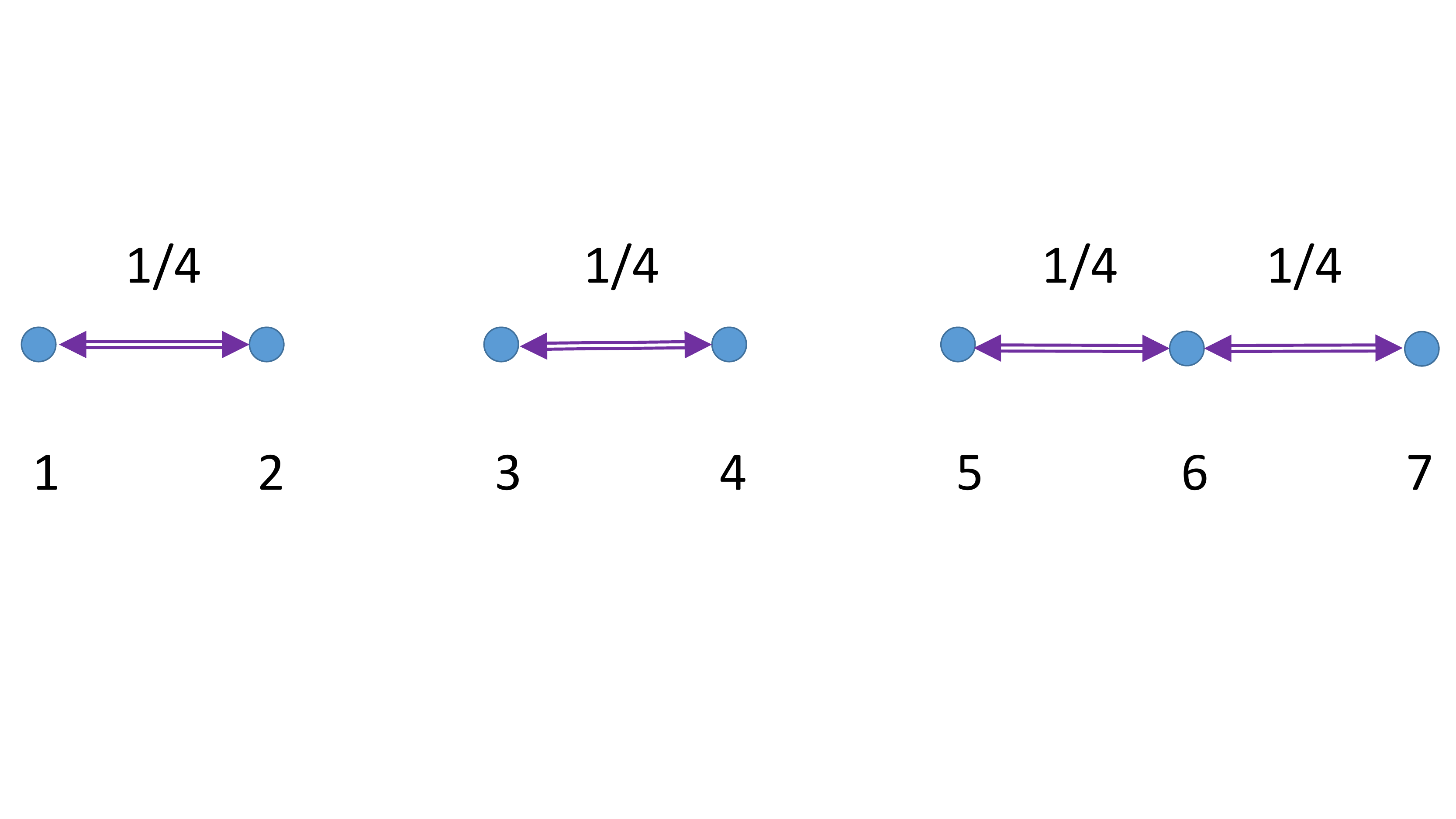}
		\caption{Unbiased covering strategy for the Patroller to oscillate on edges $\{(1,2), (3,4), (5,6), (6,7)\}$ of this minimum covering set.}
		\label{fig:example_periodic}
	\end{center}
\end{figure}

The following gives an alternative upper bound to $1/\mathcal{I}$ on $V$
based on the \textit{uniform attack strategy}, which chooses equiprobably
among the $nT$ possible attacks (pure strategies). The reason that there are
$nT$ pure strategies is because in a game with period $T$, there are $T$ periods
that the attacker can start the attack: $1,2,\ldots,T,T+1=1$, at each node.
The new upper bound is sometimes but not always better (lower) than $1/\mathcal{I}$.

\begin{proposition}
Suppose the Attacker adopts the uniform strategy on a graph $Q$. Then no Patroller
pure strategy $w(t)$ can intercept more than $2T$ of the Attacker's pure strategies,
and no more than $2T-1$ of them if $T$ is odd and $Q$ is bipartite. So,
$$ V \leq \frac{2}{n} \text{    and    } V \leq \frac{2T-1}{nT} \text{    if $T$ odd, $Q$ bipartite.} $$
\label{prop:uniform_attacker}
\end{proposition}

\begin{proof}
If $w(t)=i$ and $w(t+1)=j \neq i$ then in these two periods $w$ can intercept at most four pure Attacker strategies, namely $[i,(t-1,t)]$, $[i,(t,t+1)]$ and $[j,(t,t+1))]$, $[j,(t+1,t+2))]$, so 2 in each period and $2T$ in all. If $i=j$ then only the three attacks $[i,(t-1,t)]$, $[i,(t,t+1)]$ and $[i,(t+1,t+2)]$ can be intercepted. But if $T$ is odd and $Q$ is bipartite then $w(t) = w(t+1)$ for some $t$, so at most $2T-1$ attacks can be intercepted. Since there are $nT$ possible attacks, we have $V \leq \frac{2T}{nT}=\frac{2}{n}$ and $V \leq \frac{2T-1}{nT}$ if $T$ is odd and $Q$ is bipartite.
\end{proof}

Note that it follows from the proof of Proposition~\ref{prop:uniform_attacker} that against the
uniform attack strategy, the interception probability will be strictly less
than $2/n$ for any Patroller walk which repeats a node. This observation can be used to show that in some
cases oscillations on an edge cannot be optimal. Consider the triangle graph
shown in Figure~\ref{fig:triangle}, with $T=3.$ If the Patroller adopts a random cyclic
patrol, he intercepts any attack with probability $2/3.$ Similarly,
Proposition \ref{prop:uniform_attacker} shows that the uniform attack strategy is intercepted by any
walk with probability not exceeding 2/3, and so $V=2/3.$ On the other hand,
if the Patroller uses oscillations on edges (or any walks other than the
cycles), then he has repeated vertices and by the above remark cannot
achieve interception probability $2/3.$ So this example shows that in
general, the Patroller cannot restrict to walks restricted to individual
edges.

\begin{figure}[ht]
	\begin{center}
		\includegraphics[scale=0.3]{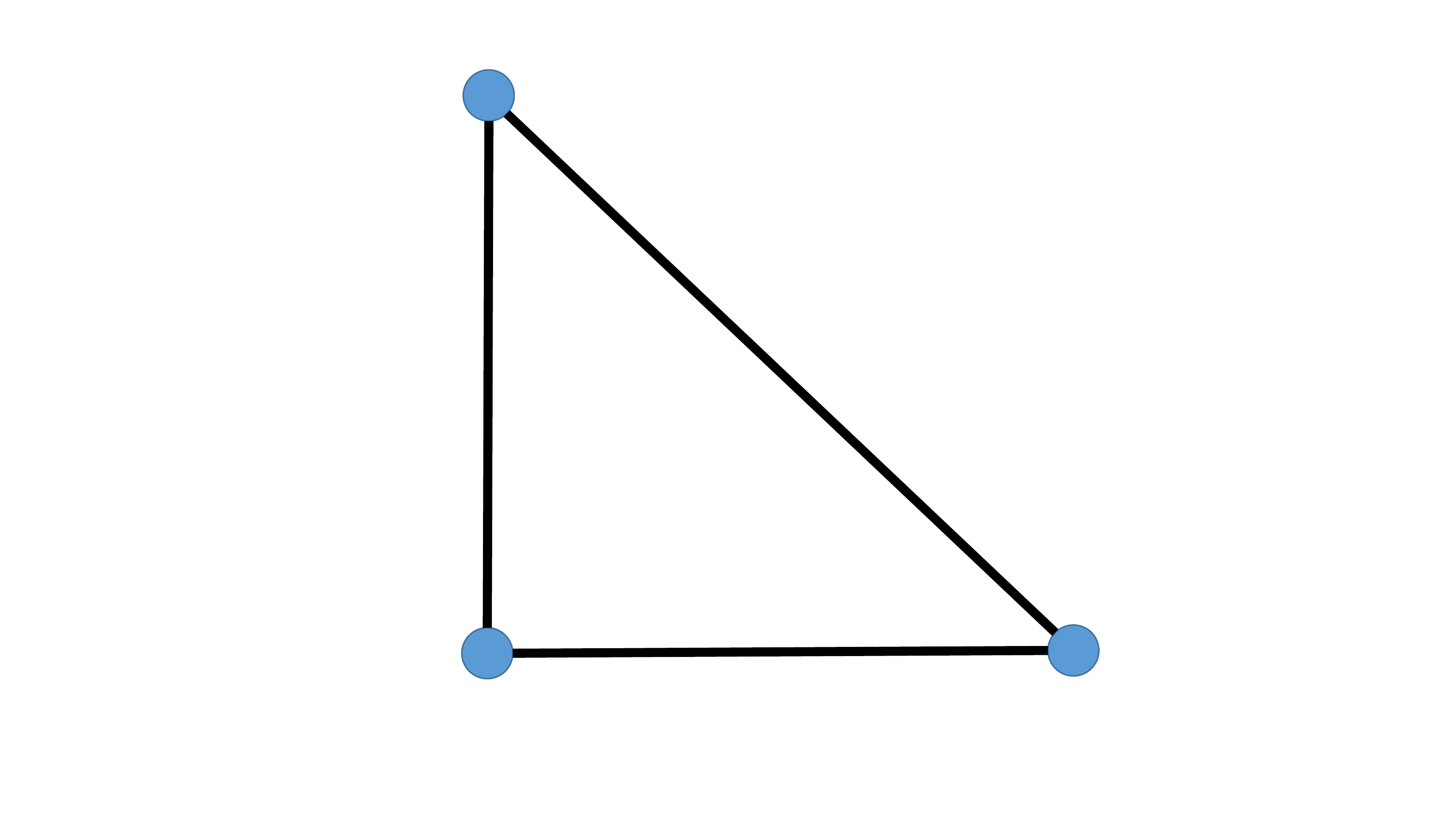}
		\caption{The triangle graph}
		\label{fig:triangle}
	\end{center}
\end{figure}

The following situation will be important in analyzing the patrolling game
on the $n$ node line graph $L_{n}$ with $n$ even. For example, consider the edge
covering of $L_4$ consisting of the edges $(1,2)$ and $(3,4)$ with $\mathcal{C} = 2 = n/2$. The covering
edges are disjoint, unlike the graph of Figure~\ref{fig:example_periodic}.

\begin{proposition}
Suppose $T$ is odd, $n$ is even and let $Q$ be a bipartite graph with $\mathcal{C}=n/2$. Then
$$V=\left( 2T-1\right) /\left(nT\right).$$
\label{prop:general_T_odd_n_even}
\end{proposition}

\begin{proof}
Since $\mathcal{C}=n/2,$ we have from~(\ref{VToddT}) that
\[
V\geq \frac{\left( 2T-1\right) }{2\mathcal{C}T}=\frac{2 \left(
2T-1\right) }{2nT}=\frac{2T-1}{nT}
\]
The result follows since for odd $T$ we have from Proposition~\ref{prop:uniform_attacker} that
$V\leq \left( 2T-1\right) /\left( nT\right) $.
\end{proof}

\subsection{Even Periods $T$}

When the period $T$ is even, we can solve the patrolling game on any graph
$Q=Q\left(  N,E\right)  $ (where $N$ is the set of nodes and $E$ is the set of edges of $Q$) by extending the notions of covering and
independence numbers to fractional forms. A more explicit solution for even
$T$ will be obtained later for line graphs.

Let $\mu:E\rightarrow\left[  0,1\right]  $ assign \textit{edge weights}
$\mu\left(  e\right)  $ to every edge $e$ so that the total weight $\hat{\mu}=%
{\textstyle\sum\nolimits_{e\in E}}
\mu\left(  e\right)  $ is minimized subject to the condition that for every
node $i\in N$ the weights $\mu\left(  e\right)  $ of the edges $e$ incident to
$i$ sum to at least $1.$ Such a $\mu$ is called an \textit{optimal edge
weighting} and $\hat{\mu}$ is called the \textit{fractional covering number}.

Similarly let $\nu:N\rightarrow\left[  0,1\right]  $ assign node weights
$\nu\left(  i\right)  $ to every node $i$ so that the total weight $\hat{\nu}=%
{\textstyle\sum\nolimits_{i}}
\nu\left(  i\right)$ is maximized subject to the condition that sum of the
weights $\nu\left(  i\right)$ of the two endpoints $i$ of every edge
$e$ is at most 1. Such a $\nu$ is called an \textit{optimal node weighting}
and $\hat{\nu}$ is called the \textit{fractional independence number}. It is
well known that $\hat{\mu}=\hat{\nu}$, a result that follows from either
duality theory or the minimax theorem applied to the game where the maximizer
picks an edge, the minimizer picks a node and the payoff is $1$ if the node is
incident to the edge and $0$ otherwise. Note that, since the number of strategies in this game is polynomial in the number of nodes of the graph, an optimal edge weighting, an optimal node weighting and $\hat{\mu}=\hat{\nu}$ can be found efficiently.
\begin{theorem}
If $T$ is even, then the value of the patrolling game is given by%
\[
V=1/\hat{\mu}=1/\hat{\nu}.
\]
An optimal strategy for the Patroller is to oscillate on edge $e$ with
probability $\mu\left(  e\right)  /\hat{\mu},$ where $\mu$ is any optimal edge
weighting. An optimal strategy for Attacker to fix any interval $\left\{
t,t+1\right\}  $ and attack at node $i$ with probabiltiy $\nu\left(  i\right)
/\hat{\nu},$ where $\nu$ is an optimal node weighting.
\end{theorem}

\begin{proof}
Suppose the Patroller chooses the stated mixed strategy and the attack is at
node $i,$ in any time interval. The Patroller will intercept the attack if he
has chosen to oscillate on an interval incident to $i,$ which has probability
at least $1/\hat{\mu}$ because the numerater is the sum of weights on edges
incident to $i.$ Similarly, suppose the Attacker adopts the stated mixed
strategy. Let $i$ and $j$ be the nodes occupied by the Patroller at the attack
times $t$ and $t+1.$ If $i\neq j,$ and $e=\left\{  i,j\right\}  $ is the edge
determined by $i\neq j$ then the probability of intercepting the attack is
given by $\nu\left(  i\right)  /\hat{\nu}+\nu\left(  j\right)  /\hat{\nu
}=\left(  \nu\left(  i\right)  +\nu\left(  j\right)  \right)  /\hat{\nu}%
\leq1/\hat{\nu}.$ If $i=j$ the same inequality holds.
\end{proof}

Note that if we restrict the weights $\mu(e)$ and $\nu(i)$ to being $0$ or $1$ we get the usual covering number $\hat{\mu} = \mathcal{C}$ and independence number $\hat{\nu} = \mathcal{I}$. Thus, from linear programming theory and duality we have:  $\mathcal{I} \leq \hat{\nu} = \hat{\mu} \leq \mathcal{C}.$

\begin{figure}[ht]
	\begin{center}
		\includegraphics[scale=0.5]{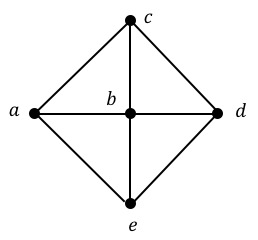}
		\caption{A non-bipartite graph}
		\label{fig:fractional-ex}
	\end{center}
\end{figure}

We consider, as an example, the graph depicted in Figure~\ref{fig:fractional-ex}. It is not bipartite, so the covering number and independence number are not equal. The covering number is $3$, and an optimal covering is $\{ ab,ac,de\}$ (where, for example $ab$ denotes the edge with endpoints $a$ and $b$). The independence number of the graph is $2$, and a maximum cardinality independent set is $\{ a, d\}$.

One optimal edge weighting is $\mu\left(  ae\right)  =1,\mu\left(  bc\right)  =\mu\left(
cd\right)  =\mu\left(  db\right)  =1/2$ and an optimal node weighting is given
by $\nu\left(  a\right)  =$ $\nu\left(  b\right)  =\nu\left(  c\right)
=\nu\left(  d\right)  =\nu\left(  e\right) = 1/2 .$ Hence $\hat{\mu} = \hat{\nu} = 5/2$. This translates to an optimal Patroller strategy that oscillates on $ae$ with probability $\mu(ae)/ \hat{\mu} = 2/5$, and oscillates on $bc$, $cd$ or $bd$ each with probability $\mu(bc)/\hat{\mu} = 1/5$. And it translates to an optimal Attacker strategy of attacking at node $i$ with probability $\nu(i)/\hat{\nu} = 1/5$, which is equivalent to the uniform Attacker strategy. We have $V=1/\hat{\mu} = 1/\hat{\nu} = 1/(5/2) = 2/5$.

\subsection{Patroller decomposition}
\label{sec:decomp}

As observed earlier in Alpern {\em et al.} (2011) the Patroller has the option of decomposing the given graph $Q$ into subgraphs $Q_{1}$ and $Q_{2}$ and randomly choosing whether to play an optimal patrolling strategy on $Q_{1}$ or on $Q_{2}$. Specifically, suppose we write the node set $N$ of $Q$ as the
(not necessarily disjoint) union $N_{1}\cup N_{2},$ and define $Q_{i}$ to be
the graph with nodes $N_{i}$ and edges between nodes that are adjacent in $Q$. Let $V_{i}$ denote the value of the patrolling game on $Q_{i}$ (with the
same parameters as on $Q$). If the Patroller optimally patrols on $Q_{i}$
with probability $p_{i},$ then any attack on a node in $Q_{i}$ will be
intercepted with probability at least $p_{i}V_{i}$. If the Patroller
equalizes these two probabilities ($p_1 V_1 = p_2V_2$) by choosing $p_{1}=V_{2}/\left(
V_{1}+V_{2}\right)$, then he wins with probability at least%
\begin{align}
p_{2}V_{2} &=p_{1}V_{1}=\frac{V_{1}V_{2}}{V_{1}+V_{2}},\text{ and hence we
have } \nonumber\\
V &\geq \frac{V_{1}V_{2}}{V_{1}+V_{2}}.  \label{eq:decomp}
\end{align}

The right-hand side of (\ref{eq:decomp}) represents the highest interception probability that
the Patroller can obtain by restricting patrols to one of the two subgraphs $%
Q_{1}$ or $Q_{2}.$ So if strict inequality holds in (\ref{eq:decomp}) then it is
suboptimal for the Patroller to decompose $Q$ in this way. If (\ref{eq:decomp}) holds with
equality, we say that the patrolling game on $Q$ with period $T$ is \textit{%
decomposable}. Note that if the game for $Q,T$ is decomposable this means
that removing edges (or barring the Patroller from using them) connecting
nodes in $Q_{1}$ to nodes in $Q_{2}$ does not lower the value of the game.

This derivation is simpler than that given in Alpern et al. (2011). We will use this
method to solve one of the cases for the line graph in Section~\ref{sec:decomposed}.

Consider the example in Figure \ref{fig:fractional-ex}. Take $N_{1}=\left\{ a,e\right\} $ and $N_{2}=\left\{ b,c,d\right\} .$ We
have $V_{1}=1$ (an oscillation intercepts any attack at $a$ or $e)$ and $%
V_{2}=2/3,$ as shown in the analysis of the triangle graph in Figure~\ref{fig:triangle}. Using the
decomposition result (5), we have
\begin{eqnarray*}
V &\geq &\frac{V_{1}~V_{2}}{V_{1}+V_{2}}=\frac{2/3}{1+2/3}=\frac{2}{5}\,%
\text{and Proposition 3 gives} \\
V &\leq &\frac{2}{n}=\frac{2}{5},\text{ so }V=2/5,\text{  }
\end{eqnarray*}%
as shown earlier in the analysis of Figure~\ref{fig:fractional-ex}, using different methods.

\section{The Line Graph}
\label{sec:line}

We now concentrate our attention on the line graph $L_{n}$ with node set $%
N=\left\{ 1,2,\dots ,n\right\} $ and edges between consecutive numbers. This graph is bipartite, with the two node sets made up of the odd
numbers and the even numbers. As mentioned in Proposition~\ref{prop:cover-indep}, this implies
that $\mathcal{I}=\mathcal{C},$ and we may take the odd numbered nodes as a
maximum independent set, giving
\begin{equation}
\mathcal{I}=\mathcal{C=}\left\{
\begin{array}{cc}
\frac{n}{2}, & \text{if }n\text{ is even, and} \\
& \\
\frac{n+1}{2}, & \text{if }n\text{ is odd.}%
\end{array}%
\right.   \label{eq:C value}
\end{equation}%
The solution of the periodic patrolling game on the line
breaks up into five cases, as outlined in Table \ref{table:summary of results}. For the Attacker
the strategies are simpler and have been defined earlier. However, for the Patroller
the strategies are more complicated and specific details for some of them can be found
at the corresponding propositions.

\begin{table}[H]
{\small
\begin{tabular}{|c|c|c|c|c|}
\hline
\text{Case} & \text{Description} & \text{Value} & \text{Patroller strategy}
& \text{Attacker strategy} \\ \hline
1 & \text{$T,n$ even} & $\frac{2}{n}$ & \text{unbiased covering strategy} & \text{%
independent} \\
& Proposition~\ref{prop:cases1-3} &  &  Lemma~\ref{lemma:cover-independence}&  Lemma~\ref{lemma:cover-independence}\\ \hline
2 & \text{$T$ even, $n$ odd} & $\frac{2}{n+1}$ & \text{unbiased covering strategy} &
\text{independent} \\
& Propostion~\ref{prop:cases1-3} &  &  Lemma~\ref{lemma:cover-independence} &  Lemma~\ref{lemma:cover-independence}\\ \hline
3 & \text{$T$ odd, $n$ even} & $\frac{2T-1}{nT}$ & biased covering strategy & \text{uniform} \\
&  Proposition~\ref{prop:cases1-3} &  & Proposition~\ref{prop:cover-indep} &  Proposition~\ref{prop:uniform_attacker} \\ \hline
4 & \text{ $T,n$ odd, }$n \geq 2T-1$ & $\frac{2T-1}{nT}$ & mixture of
$p$-biased oscillations &
uniform \\
&  Propositions~\ref{prop:case_4},~\ref{prop:case_4-decomp}&  & (Prop~\ref{prop:case_4}) or decomposed (Prop~\ref{prop:case_4-decomp}) & Prop~\ref{Discrete Line value u.b.},~Fig~\ref{fig:graph} \\ \hline
5 & \text{$T,n$ odd, }$n \leq 2T-1$ & $\frac{2}{n+1}$ & mixture of $p$-biased oscillations & independent \\
&  Proposition~\ref{prop:case_5}&  & Proposition~\ref{prop:case_5} &Prop~\ref{Discrete Line value u.b.},~Fig~\ref{fig:graph}  \\ \hline
\end{tabular}%
} \\
\caption{Solution of Patrolling Game on $L_n$, period $T$.}
\label{table:summary of results}
\end{table}

We give below in Figure \ref{fig:matrix_cases} a partition of $(n,T)$ into the five cases of Table \ref{table:summary of results}. The pattern is quite complicated.

\begin{figure}[H]
	\begin{center}
		\includegraphics[scale=0.3]{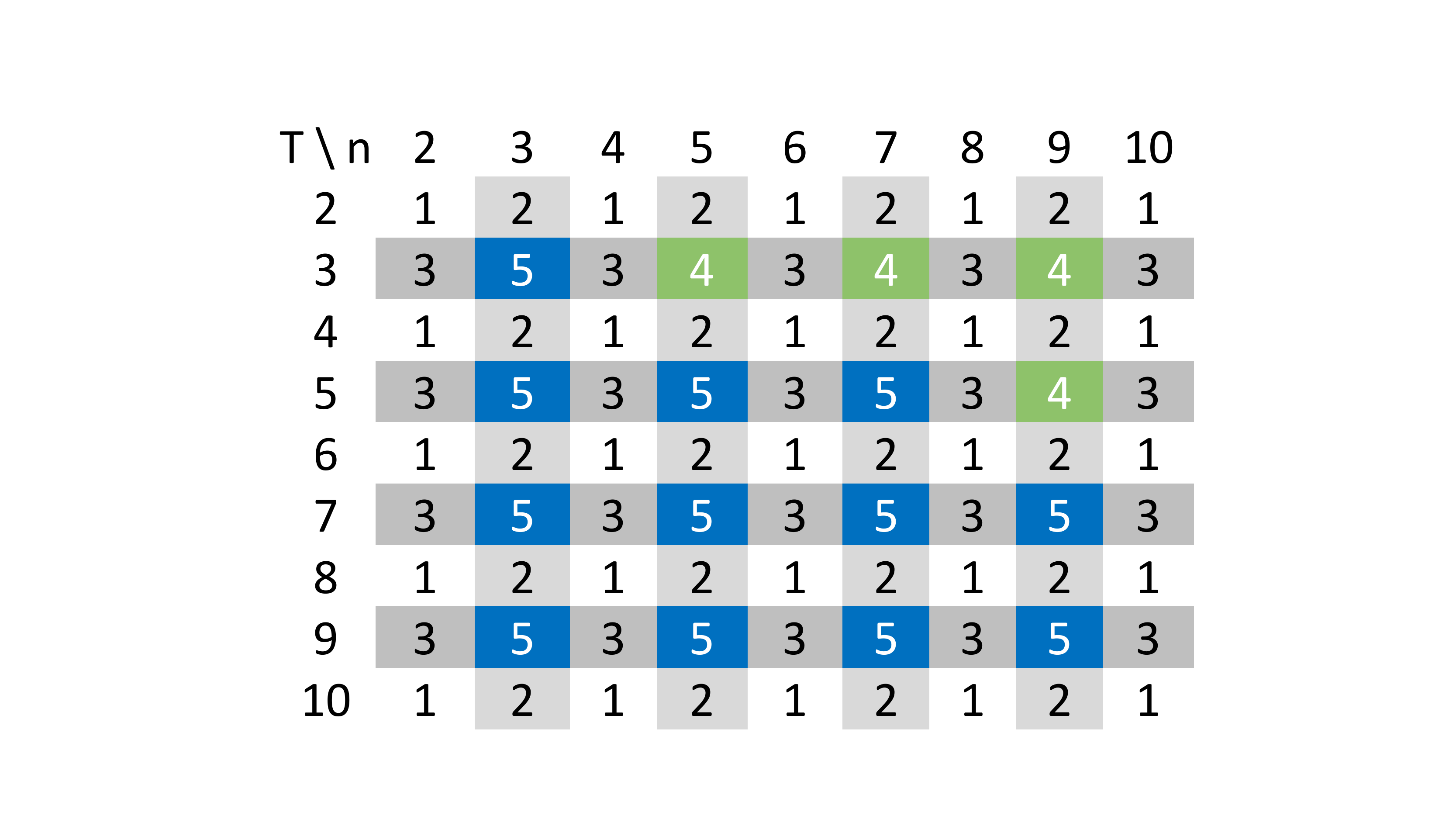}
		\caption{Cases from Table \ref{table:summary of results} for pairs of $(n,T)$.}
		\label{fig:matrix_cases}
	\end{center}
\end{figure}

\subsection{Cases 1 to 3 (one of $T$ or $n$ is even)}

If either $T$ or $n$ is even, there are three different forms for the value,
but all follow easily from previous results.

\begin{proposition}
\label{prop:cases1-3} For $L_n$, if $T$ is even, then
\begin{equation}
V=\frac{1}{\mathcal{C}}=\left\{
\begin{array}{cc}
\frac{2}{n} & \text{if }n\text{ is even,} \\
& \\
\frac{2}{n+1} & \text{if }n\text{ is odd.}%
\end{array}%
\right.  \label{eq:value_T_even}
\end{equation}%
If $T$ is odd and $n$ is even we have
\begin{equation}
V=\frac{2T-1}{nT}.
\label{eq:Toddneven}
\end{equation}
\end{proposition}

\begin{proof}

First suppose that $T$ is even. In this case, the result~(\ref{eq:value_T_even}) easily follows from Proposition~\ref{prop:cover-indep} and (\ref{eq:C value}), since $L_n$ is bipartite.

For $T$ odd and $n$ even, there is an edge covering of $L_{n}$ with $\mathcal{C}$ $=n/2$ disjoint
edges of the form $\left\{ 2i-1,2i\right\} ,$ $i=1,\dots ,n/2$. Thus the result follows from Proposition~\ref{prop:general_T_odd_n_even}.
\end{proof}

Thus the only remaining cases (4 and 5) are when $T$ and $n$ are both odd. These are the complicated cases.

\subsection{Comparison of uniform and independent attack strategies}

For the remaining cases when $T$ and $n$ are both odd, we must compare the effectiveness of two
different strategies for the Attacker: the uniform strategy, mentioned
above, chooses equiprobably among all the $nT$ possible pure stategies (at
all $n$ nodes at all $T$ starting times); the independent strategy starts at
time, say, $1$ and chooses equiprobably among the $\mathcal{I}$ independent
nodes. That is, the independent strategy chooses among $\mathcal{I}$ {\em simultaneous} attacks.
We have already obtained two different upper bounds on $V$ for these
cases: $\left( 2T-1\right) /\left( nT\right) $ from Proposition ~\ref{prop:uniform_attacker}, for the uniform attack strategy; and $2/\left(n+1\right) $ from~(\ref{VToddT}), for the independent strategy (since $\mathcal{I}=\frac{n+1}{2}$). In
general, neither of these is the better (lower) bound, as can be seen in
Figure~\ref{fig:graph}.

\begin{figure}[H]
	\begin{center}
		\includegraphics[scale=0.3]{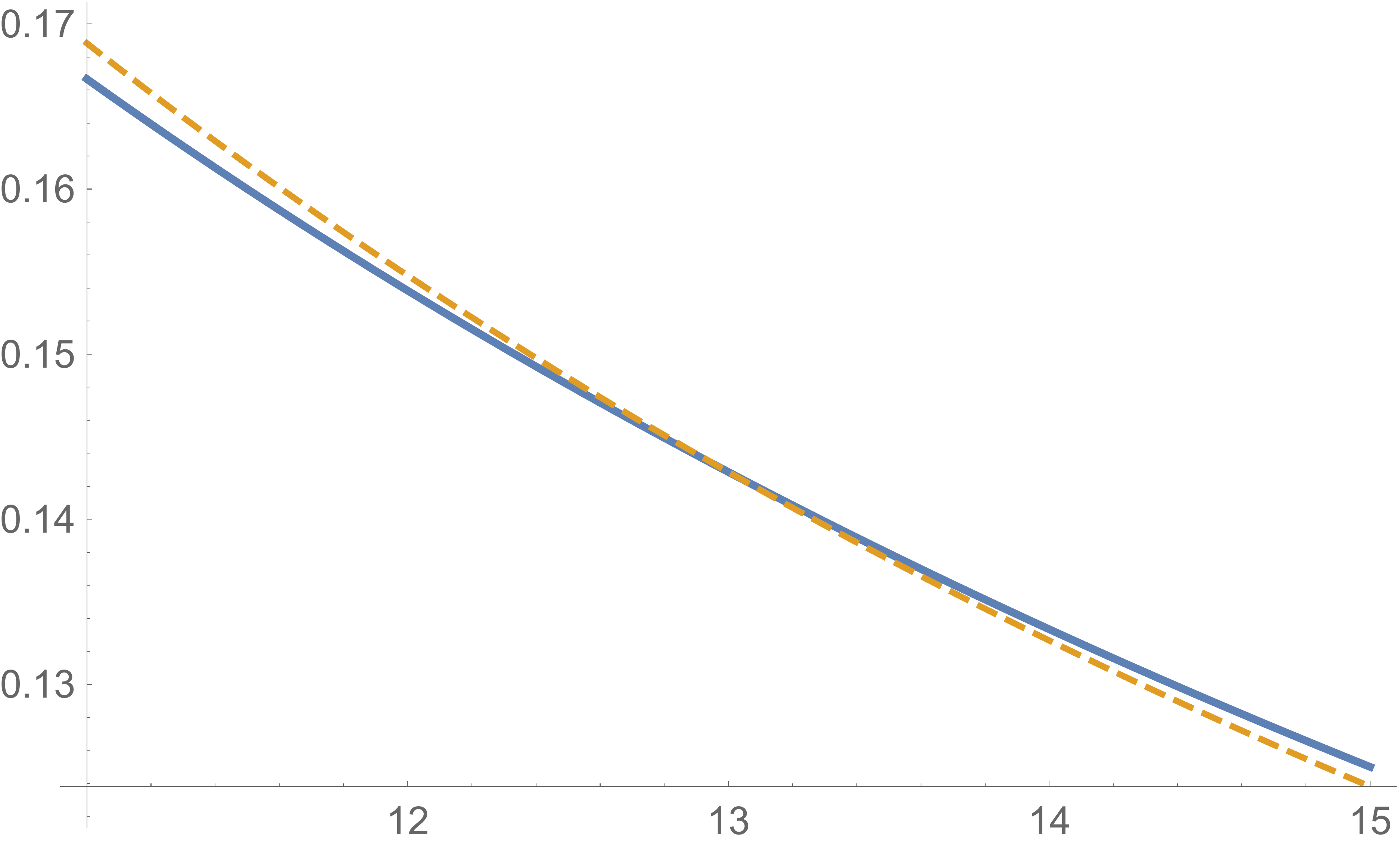}
		\caption{Plots of $2/(n+1)$ (solid) and $(2T-1)/(nT)$ (dashed), $T=7$, $n=11,\ldots,15$.}
		\label{fig:graph}
	\end{center}
\end{figure}

Note that the two curves intersect at $n=2T-1$ (at $n=13$ in the figure). Since the Attacker can choose the attack (uniform or independent) which gives the smaller upper bound on the value, we can summarize his options as follows.

\begin{proposition}
\label{Discrete Line value u.b.} Suppose $T$ and $n$ are both odd, and $%
Q=L_{n}$. Then%
\begin{equation*}
V\leq \min \left( \frac{2T-1}{nT},\frac{2}{n+1}\right) =\left\{
\begin{array}{cc}
\frac{2}{n+1} & \text{if }n\leq 2T-1, \\
& \\
\frac{2T-1}{nT} & \text{if }n\geq 2T-1.%
\end{array}%
\right.
\end{equation*}
\end{proposition}

We now analyze these two cases for $n$ separately, beginning with $n\leq
2T-1 $. For the Patroller strategies we shall use oscillations which are similar to the walks $w_{k}$
which appeared in the proof of Proposition~\ref{prop:cover-indep}.

\subsection{Case 4 (\text{$T,n$ odd, }$n \geq 2T-1$)}
\label{sec:case4}

To deal with the case of $n \geq 2T-1 $ and noting the the oddness of $T$
requires a stunted type of oscillation, we define {\em $p$-biased oscillations} as follows.

\begin{definition}
For $p \in [0,1]$, a {\bf right $p$-biased oscillation} $\overrightarrow{b}_p(i)$ (for $%
i=1,\ldots ,n-1$) is a $T$-periodic walk between $i$ and $i+1$ where $i$ and
$i+1$ alternate except that with probability $p$, at a random time, the right-hand node $i+1$
is repeated (if $T=2q+1$, it is at node $i+1$ for $q+1$ periods and at $i$
for $q$ periods); with probability $1-p$, at a random time, the left-hand node is repeated. For convenience, we define a {\bf left $p$-biased  oscillation} $\overleftarrow{%
b}_p(i)$ as $\overrightarrow{b}_{1-p}(i)$. If $p=1/2$, we will refer to a right (or left) $p$-biased oscillation as an {\bf unbiased oscillation}.
\end{definition}

For the following result note that for larger $n$ the uniform attack strategy is better for the Attacker than the independent attack strategy.

\begin{proposition}
For $L_n$, assume that both $T$ and $n$ are odd and that $2T \le n+1$. Then
\begin{equation*}
V=\frac{2T-1}{nT}.
\end{equation*}%
The uniform attack strategy is optimal for the Attacker and a probabilistic choice of biased oscillations is optimal for the Patroller.
\label{prop:case_4}
\end{proposition}
The reader is invited to read the example in Table \ref{tab:L7-case4} and commentary to obtain some intuition for the proof.\\
\begin{proof}
From Proposition \ref{Discrete Line value u.b.} we know that $V\leq \frac{2T-1}{nT}$, so it is enough to demonstrate a Patroller strategy
which intercepts an attack at any node $i$ with probability at least $\frac{2T-1}{nT}$.

For $j=1,\ldots,(n+1)/2$, let $A_j$ be the set of edges of the form $(2i-1,2i)$ for $i < j$. For example, $A_1$ is empty and $A_3 = \{(1,2), (3,4)\}$. Also let $B_j$ be the set of edges of the form $(2i,2i+1)$ for $i \ge j$, so $B_1 = \{(2,3),(4,5),\ldots,(n-1,n)\}$ and $B_3 = \{(6,7),(8,9),\ldots,(n-1,n)\}$. Finally let $D_j = A_j \cup B_j$.

For example when $n=7$ we have $A_{2}=\left\{ (1,2)\right\},$ $B_{2}=\left\{ (4,5), (6,7)\right\}$ and
$D_{2}=\left\{(1,2), (4,5), (6,7)\right\},$ as shown by the three arrows (for edges) on the second line from the top in Table~\ref{tab:L7-case4}. The arrows are oriented left for edges in $A_{2}$ and
right for those in $B_{2}$ to indicate the Patroller's use of left or right
biased oscillations on these edges in his optimal strategy.

There are $(n-1)/2$ edges in $D_j$, and each node in the line graph except one is incident to some edge in $D_j$, for each $j$.

Consider the following Patroller strategy. First some $j$ is chosen uniformly at random,\\ $j=1,\ldots, (n+1)/2$ and an edge $(i,i+1)$ in $D_j$ is chosen uniformly at random. If $(i,i+1)$ is contained in $A_j$ then the Patroller performs a left $p$-biased oscillation $\overleftarrow{b}_p(i)$. If $(i,i+1)$ is in $B_j$ then the Patroller performs a right $p$-biased oscillation $\overrightarrow{b}_{p}(i)$. This probability $p$ will be determined later.

If a node is either on the left of an edge in some $A_j$ that is being patrolled or if it is on the right of an edge in some $B_j$ that is being patrolled, then an attack at that node is intercepted with probability:
\begin{equation}
p \cdot 1 + (1-p) \cdot(T-1)/T = (T+p-1)/T.
\label{eq:forward}
\end{equation}
If a node is either on the {\em right} of an edge in some $A_j$ that is being patrolled or if it is on the {\em left} of an edge in some $B_j$ that is being patrolled, then an attack at that node is intercepted with probability:
\begin{equation}
p \cdot (T-1)/T + (1-p) \cdot 1 = (T-p)/T.
\label{eq:backward}
\end{equation}
We first calculate the probability $p_{2i}$ that an attack at an even numbered node $2i$ is intercepted, $i=1,\ldots, (n-1)/2$. Observe that for every one of the $(n+1)/2$ values of $j$, the node $2i$ is either on the right of an edge in $A_j$ or on the left of an edge in $B_j$, so
\begin{align}
p_{2i} = \left( \frac{1}{(n-1)/2} \right)  \left(\frac{T-p}{T} \right) = \frac{2(T-p)}{(n-1)T}. \label{eq:p2i}
\end{align}
For an odd numbered node $2i-1,i=1,\ldots,(n+1)/2$, we observe that there are $(n-1)/2$ values of $j$ such that the node $2i-1$ is either on the left of an edge in $A_j$ or on the right of an edge in $B_j$. There is one value of $j$ such that node $2i-1$ is not incident to any edge in $A_j$ or $B_j$. So the probability $p_{2i-1}$ that an attack at node $2i-1$ is intercepted is
\begin{align}
p_{2i-1} = \left( \frac{1}{(n-1)/2} \right) \left( \frac{(n-1)/2}{(n+1)/2} \right) \left(\frac{T+p-1}{T} \right) = \frac{2(T+p-1)}{(n+1)T}. \label{eq:p2i+1}
\end{align}
Since $2T \le n+1$, we may choose $p=(2T+n-1)/(2n)$ so that the probabilities $p_{2i}$ and $p_{2i-1}$ are equal, and substituting this value of $p$ into~(\ref{eq:p2i}) or~(\ref{eq:p2i+1}), we obtain the bound
\[
V \ge \frac{2(T-(2T+n-1)/2n)}{(n-1)T} = \frac{(2T-1)}{nT}.
\]
Combining this with our lower bound, this establishes the proposition.
\end{proof}

We illustrate the Patroller's optimal strategy, taking $L_7$ as an example, with $T=3$ in Table \ref{tab:L7-case4}. The four choices of $D_1,\ldots,D_4$ correspond to the four rows in Table \ref{tab:L7-case4}. The left pointing arrows correspond to the edges in the $A_j$ and the right pointing arrows correspond to the edges in the $B_j$. Nodes which are incident to one of the edges in $D_j$, are indicated by a solid disk, those which are not, by an outlined disk.

\begin{table}[htb!]
\begin{gather*}
\begin{tabular}{llllllllllllll}
1 &  & 2 &  & 3 &  & 4 &  & 5 &  & 6 &  & 7 & \\
$\circ $ &  & $\bullet $ & $\Longrightarrow $ & $\bullet $ &  & $\bullet $ &
$\Longrightarrow $ & $\bullet $ &  & $\bullet $ & $\Longrightarrow $ & $%
\bullet $ & $D_1$\\
$\bullet $ & $\Longleftarrow $ & $\bullet $ &  & $\circ $ &  & $\bullet $ & $%
\Longrightarrow $ & $\bullet $ &  & $\bullet $ & $\Longrightarrow $ & $%
\bullet $ & $D_2$\\
$\bullet $ & $\Longleftarrow $ & $\bullet $ &  & $\bullet $ & $%
\Longleftarrow $ & $\bullet $ &  & $\circ $ &  & $\bullet $ & $%
\Longrightarrow $ & $\bullet $ & $D_3$\\
$\bullet $ & $\Longleftarrow $ & $\bullet $ &  & $\bullet $ & $%
\Longleftarrow $ & $\bullet $ &  & $\bullet $ & $\Longleftarrow $ & $\bullet
$ &  & $\circ $ & $D_4$\\
\end{tabular}
\end{gather*}%
	\caption{Optimal strategy for $L_{7}$ with $T=3$.}
	\label{tab:L7-case4}
\end{table}

The Patroller picks one of the rows of the table at random, and then one of the arrows in that row at random, corresponding to an edge $(i,i+1)$.  Equivalently, he picks one of the $12$ arrows at random. Then he performs a left or right $p$-biased oscillation, depending on the direction of the arrow, where $p=(2T+n-1)/(2n)=12/14=6/7$.  If a node has three arrows pointing toward it (odd nodes), then an attack at that node is intercepted with probability $(3/12)(p+(1-p)(T-1)/T) = (1/4)(6/7+(1/7)(2/3))=5/21$. If, on the other hand, a node has four arrows pointing away from it (even nodes), then an attack at that node is intercepted with probability $(4/12)((1-p)+p(T-1)/T) = (1/3)(1/7+(6/7)(2/3)) = 5/21$. So the value is $5/21 = (2T-1)/(nT)$.





\subsection{Case 5 (\text{$T,n$ odd,}$n \leq 2T-1$)}

We now consider the remaining open case of $n$ and $T$ odd and $n \leq 2T-1$.

\begin{proposition}
		\label{prop:case_5}
	For $L_n$, assume that both $T$ and $n$ are odd and that $n \leq 2T-1$. Then
	\begin{equation*}
	V=\frac{2}{n+1}.
	\end{equation*}%
	The independent strategy is optimal for the Attacker and a probabilistic choice of biased oscillations is optimal for the Patroller.
\label{prop:case5_Patroller}
\end{proposition}

\begin{proof}
	It follows from Proposition~\ref{Discrete Line value u.b.} that $V \le 2/(n+1)$. To prove the reverse bound on the value, we simply use the Patroller strategy described in the proof of Proposition~\ref{prop:case_4}, but this time taking $p=1$ in Equations~(\ref{eq:p2i+1}) and~(\ref{eq:p2i}) to obtain
	\[
	p_{2i-1} = \frac{2(T+p-1)}{(n+1)T} = \frac{2}{n+1} \mbox{ and } p_{2i} = \frac{2(T-1)}{(n-1)T} \ge \frac{2}{n+1},
	\]
where the last inequality follows directly from $n \leq 2T-1$. Thus, we have $V \geq 2/(n+1)$.
\end{proof}

\subsection{Decomposed strategies}
\label{sec:decomposed}

We may now also give an alternative optimal strategy for the Patroller in case 4, using a decomposition of the line graph.

\begin{proposition}
For $L_n$, if $T$ and $n$ are odd and $n>2T-1$ then $V=\frac{2T-1}{nT}$.
The uniform strategy is optimal for the Attacker. For the Patroller there is
an optimal strategy which decomposes the graph $Q=L_{n}$ into a left graph $%
\mathcal{L}$ $=L_{n_{\mathcal{L}}}$ with the odd number $n_{\mathcal{L}}=2T-1
$ of nodes $\left\{ 1,2,\dots ,2T-1\right\} $ and a right graph $\mathcal{R=}%
L_{n_{\mathcal{R}}}$with the remaining even number $n_{\mathcal{R}}=n-\left(
2T-1\right) $ of nodes $\left\{ 2T,2T+1,\dots ,n\right\} .$
\label{prop:case_4-decomp}
\end{proposition}

\begin{proof}
The adoption of the uniform attacker strategy guarantees that $V\geq
\frac{2T-1}{nT}$ by Proposition~\ref{prop:uniform_attacker}.  The left graph $\mathcal{L}$ satisfies the
hypotheses of Proposition~\ref{prop:case5_Patroller},  because $n_{\mathcal{L}} \leq 2T-1$ (equality
holds) and $T$ and $n_{\mathcal{L}}$ are odd. Hence Proposition~\ref{prop:case5_Patroller} gives%
\[
V\left( \mathcal{L}\right) =\frac{2}{n_{\mathcal{L}}+1}=\frac{2}{2T}.
\]%
The subgraph $\mathcal{R}$ has an even number of nodes $n_{\mathcal{R}},$ so
it satisfies the hypothesis of Proposition~\ref{prop:cases1-3}, hence equation (\ref{eq:Toddneven}) gives
\[
V\left( \mathcal{R}\right) =\frac{2T-1}{n_{\mathcal{R}}T}=\frac{2T-1}{\left(
n-\left( 2T-1\right) \right) T}.
\]%
It follows from the decomposition estimate (\ref{eq:decomp}) that
\[
V=V\left( L_{n}\right) \geq \frac{V\left( \mathcal{L}\right) ~V\left(
\mathcal{R}\right) }{V\left( \mathcal{L}\right) +V\left( \mathcal{R}\right) }%
=\frac{2T-1}{nT}.
\]
\end{proof}

As an example, consider again the case $T=3$ and $n=7>2T-1=5$, as considered in Section~\ref{sec:case4}. As we know, $V_{7}=5/21$. We decompose $L_{7}$ into
$\mathcal{L} = L_5$ and $\mathcal{R} = L_{2}$. On $L_{5},$ the optimal Patroller strategy is given by
Proposition \ref{prop:case_5}. On $L_{2}$ the optimal Patroller
strategy is an unbiased oscillation on the single edge $\left( 6,7\right) $.

According to Section~\ref{sec:decomp}, the probabilities $p_{5}$ and $%
p_{2}$ of patrolling on $L_{5}$ and $L_{2}$ should satisfy $p_{5}V_{5}=p_{2}V_{2}$. Since $V_{5}=2/(5+1)=1/3$ and $V_{2}=(2\cdot 3-1)/(2\cdot 3)=5/6$, we have $p_{5}=5/7$ and $%
p_{2}=2/7$.

We may represent this strategy by the diagram in Table~\ref{tab:L7-decomposed}, where $L_7$ is decomposed into $\mathcal{L} = L_5$ (on the left) and $\mathcal{R} = L_2$ (on the right). The Patroller first chooses $L_5$ with probability $p_5=5/7$ and $L_2$ with probability $2/7$. If he chooses $L_2$ then he performs an unbiased oscillation (indicated by the double-ended arrow) on edge $(6,7)$. If he chooses $L_5$ then he chooses one of the single-ended arrows at random and performs a left or right biased $p$-oscillation, depending on the direction of the arrow, with $p=1$.


\begin{table}[htb!]
	\begin{gather*}
	\begin{tabular}{llllllllll|lll}
	1 &  & 2 &  & 3 &  & 4 &  & 5 &  & 6 &  & 7 \\
	$\circ $ &  & $\bullet $ & $\Longrightarrow $ & $\bullet $ &  & $\bullet $ &
	$\Longrightarrow $ & $\bullet $ & & $\bullet $ & $\Longleftrightarrow $ & $%
	\bullet $ \\
	$\bullet $ & $\Longleftarrow $ & $\bullet $ &  & $\circ $ &  & $\bullet $ & $%
	\Longrightarrow $ & $\bullet $ & & $\bullet $ & $\Longleftrightarrow $ & $%
	\bullet $\\
	$\bullet $ & $\Longleftarrow $ & $\bullet $ &  & $\bullet $ & $%
	\Longleftarrow $ & $\bullet $ &  & $\circ $ & & $\bullet $ & $%
	\Longleftrightarrow $ & $\bullet $ \\
	\end{tabular}
	\end{gather*}%
	\caption{Decomposed strategy for $L_{7}$ with $T=3$.}
		\label{tab:L7-decomposed}
\end{table}

\subsection{Decomposable patrolling games}

We can now determine for which values of $T$ and $n$ the line graph $L_{n}$
is decomposable (equality in (5)), in the sense that the Patroller can
restrict his patrols to one of two disjoint subgraphs without loss of
optimality.

\begin{proposition}
The patrolling game on the line is decomposable unless $T$ and $n$ are odd
and $n\leq 2T-1$ (case 5).
\end{proposition}

\begin{proof}
First we show that for cases 1through 4 in Table 1, the patrolling game is
decomposable (by the Patroller). In cases 1, 2 and 3, the Patroller uses
what we call covering strategies, in that his pure patrols are on edges
forming a minimum covering set. For $n\geq 4,$ such as set can include the
edge $\left( 1,2\right) $ and $\left( 3,4\right) $ and in particular the
Patroller can avoid using the edge $\left( 2,3\right) .$ It follows that he
is decomposing $L_{n}$ into $\mathcal{L}=L_{2}$ and $\mathcal{R}=L_{n-2}$
with disjoint nodes sets $\left\{ 1,2\right\} $ and $\left\{ 3,\dots
,n\right\} .$ (If $T$ is even and $n=3,$ then instead of using the covering
strategy involving edges $\left( 1,2\right) $ and $\left( 2,3\right) ,$ the
Patroller decomposes the game by equiprobably oscillating on edge $\left(
1,2\right) $ and remaining stationary on node $3$ to obtain an interception
probability of $1/2=V\left( L_{3}\right) .)$ For case 4, the  optimal
Patroller strategy given in Proposition 9 does not decompose the game.
However an optimal strategy which does decompose the game is given in
Proposition 10, where $L_{n},$ $n$ odd, is decomposed into  $\mathcal{L}%
=L_{2T-1}$ and $\mathcal{R}=L_{n-\left( 2T-1\right) }.$ This is a strategy
where the Patroller never traverses the edge $\left( 2T-1,2T\right) .$

So assume that $T$ and $n$ are odd and $n\leq 2T-1$ (case 5). So any
decomposition of $L_{n}$ is into an even node line graph $L_{2j},$ $j>0$ and
an odd one $L_{n-2j}.$ The assumptions on $T$ and $n$ are covered by
Proposition $9,$ so we have%
\[
V\left( L_{n}\right) =\frac{2}{n+1}\text{.}
\]%
Since $2j$ is even, it follows from (8) in Proposition 5, that
\[
V\left( L_{2j}\right) =\frac{2T-1}{2jT}.
\]%
Since $n-2j$ is odd and $n-2j<n\leq 2T-1,$ it follows from Proposition 9 that%
\[
V\left( L_{n-2j}\right) =\frac{2}{\left( n-2j\right) +1}\text{.}
\]%
The best the Patroller can do by such a decomposition (see Section 3.2) is
to obtain an interception probability of%
\[
\frac{V\left( L_{2j}\right) \ast V\left( L_{n-2j}\right) }{V\left(
L_{2j}\right) +V\left( L_{n-2j}\right) }.
\]
The difference between the unrestricted value and the restricted one is
given above is%
\begin{eqnarray*}
V\left( L_{n}\right) -\frac{V\left( L_{2j}\right) \ast V\left(
L_{n-2j}\right) }{V\left( L_{2j}\right) +V\left( L_{n-2j}\right) } &=&\frac{2%
}{n+1}-\frac{\left( \frac{2T-1}{2jT}\right) \ast \left( \frac{2}{%
\left( n-2j\right) +1}\right) }{\left( \frac{2T-1}{2jT}\right)
+\left( \frac{2}{\left( n-2j\right) +1}\right) } \\
&=&\frac{4j}{\left( n+1\right) \left( 2j+\left( 2T-1\right) \left(
1+n\right) \right) }>0.
\end{eqnarray*}
\end{proof}

%
%
%
%
%
%
%

\subsection{Limiting values for large periods $T$}
\label{sec:limit_on_T}

Compared with games with simply a fixed time horizon $T,$ the problem with
\textit{period} $T$ is more difficult for the Patroller, as he has the
additional requirement that he has to end at the same node as he started.
However as the period gets large, this restriction is less oppressive to the
Patroller, because the amount of time he must use to get back to his start
is the fixed diameter of the graph. In this subsection we check that the
limiting value of $V\left( T,n\right) $ for the game with period $T$
approaches the value $V\left( n\right) $ found for the patrolling game on
the graph $L_{n}$ without periodic patrols. For $m=2$ the values found in
Papadaki et al. (2016) are simply $V\left( n\right) =1/\left\lceil
n/2\right\rceil ,$ that is, $2/n$ for even $n$ and $2/\left( n+1\right) $
for odd $n.$ If we look at the values $V\left( T,n\right) $ for periodic
patrols found for the five cases, looking back at Table~\ref{table:summary of results}, for cases 1, 2, 3
and 5 (case 4 does not hold as $T$ goes to infinity), we obtain the same
limiting value%
\[
\lim_{T\rightarrow \infty }V\left( T,n\right) =V\left( n\right)
=1/\left\lceil n/2\right\rceil ,\text{ for all }n.
\]

Of course this is not an easy way of establishing the nonperiodic result, as
the periodic case dealt with here is more complicated.

\subsection{Further connections with non-periodic game}
\label{sec:non-periodic}

The solution of the non-periodic game on $L_n$ as given in Papadaki el al (2016) involves periodic patrols of different periods $T_1,\ldots,T_k$. Setting $T^*$ to be the least common multiple of $\{T_1,T_2,\ldots, T_k\},$ we see that the solution has period $T^*$. If we were seeking a solution to the periodic game with set period $T^*$ the same solution would be valid.

Let us consider the example with $n=7,$ $m=2$ (in the non-periodic game there is no given $T$). The solution given there is as follows: with probability $1/8$ adopt unbiased oscillations on edges $(1,2)$ and $(6,7)$ and with probability $6/8$ adopt a tour of $L_n$ of period $2(n-1)=12$ that goes back and forth between the end nodes. This is illustrated in Figure \ref{fig:example_nonperiodic}.

\begin{figure}[H]
	\begin{center}
		\includegraphics[scale=0.3]{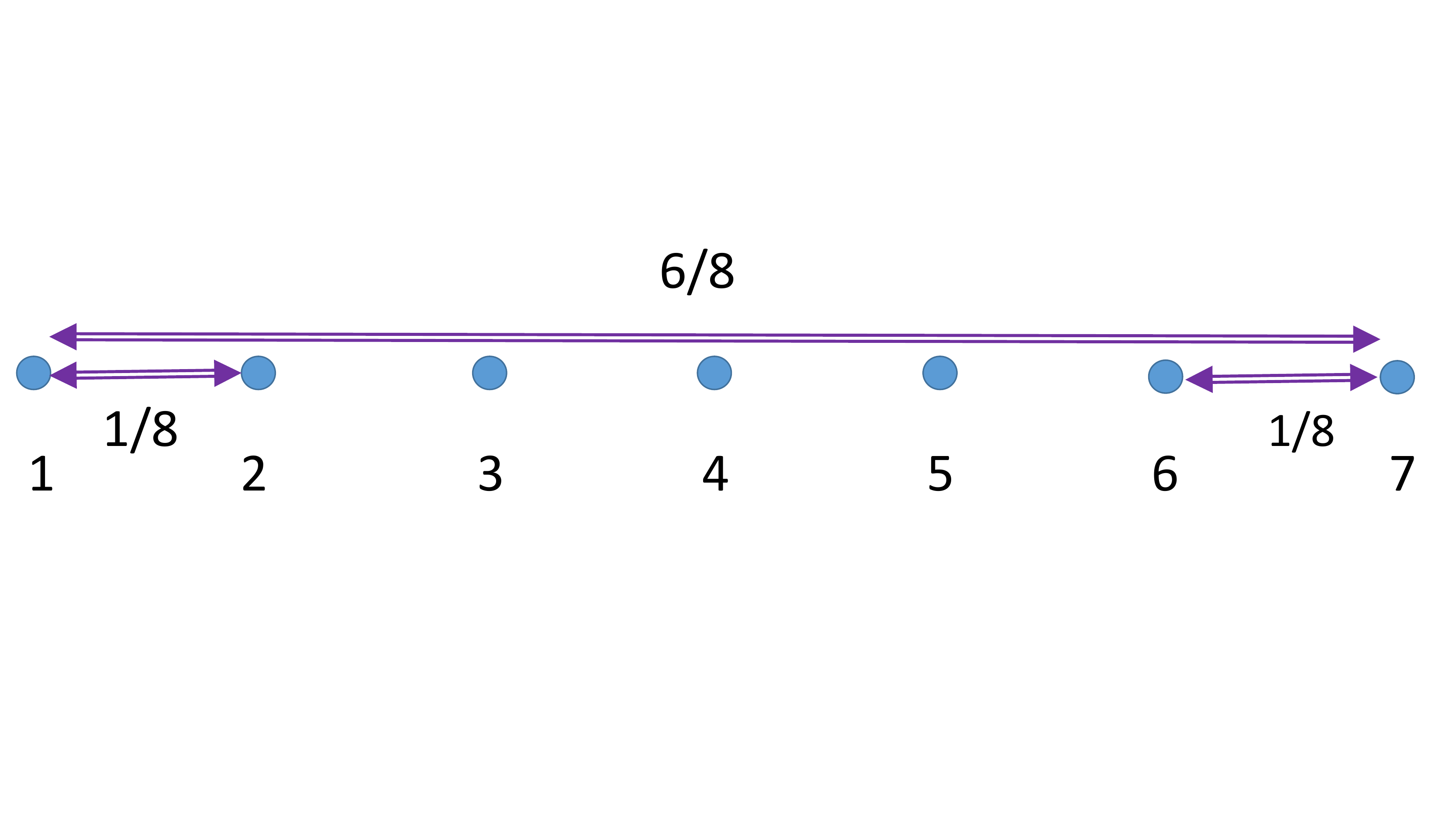}
		\caption{Patroller oscillates between end nodes with probability $6/8$ and on edges $(1,2)$ and $(6,7)$ each with probability $1/8$ in $L_7$.  }
		\label{fig:example_nonperiodic}
	\end{center}
\end{figure}

It is easy to check that the probability that the tour of $L_{n}$ (of period
$12$) intercepts attacks at nodes $1,2,\dots ,7$ is given respectively by $%
2/12,3/12,4/12,4/12,4/12,3/12,2/12$. The 12-cycle can be written, starting at
say node 3, as $3^{\ast },4^{\ast },5^{\ast },6^{\ast
},7,6,5,4,3,2,1,2^{\ast },\dots$, where $^{\ast }$ indicates going to the
right. Note that an attack at node 4 starting at time $t$ will be
intercepted if the Patroller following this cycle is at one of the four
steps $5,4$ or $3^{\ast },4^{\ast }$ out of the twelve steps in the cycle,
that is, with probability $4/12.$ The other probabilities are calculated in
a similar manner. For example the Patroller can be at steps $2^{\ast }$ or $1
$ to intercept an attack at node $1$ and at steps $2,1$ or $2^{\ast }$ to
intercept an attack at node $2$.

We now calculate the probability that the mixed strategy stated above intercepts an attack at each node. For node $1$ such an attack is intercepted with probability $2/12$ by the big oscillation and with probability $1$ by the oscillation on edge $(1,2)$. Hence, the total interception probability is given by $(6/8)(2/12)+(1/8)(1)=1/4$. An attack in node $2$ is intercepted with probability $3/12$ by the big oscillation and with probability $1$ by the oscillation on edge $(1,2)$. Hence, the total interception probability is given by $(6/8)(3/12)+(1/8)(1)=5/16$. At node $3$ an attack is intercepted with probability $4/12$ by the big oscillation. Hence the total interception probability is given by $(6/8)(4/12)=1/4$. The argument for node $4$ is the same as node $3$ and the interception probabilities for nodes $5,6,7$ are the same as nodes $3,2,1$ respectively by symmetry. So the overall interception probabilities for nodes $\{1,2,...,7\}$ are given by $\{1/4,5/16,1/4,1/4,1/4,5/16,1/4\}$. The minimum is $1/4$, which is also the value of $m/(n+m-1) =  1/4$, given by Papadaki el al (2016). Note that the Attacker can achieve a successful attack with probability $1/4$ by attacking equiprobably simultaneously at the nodes of the independent set $\{1,3,5,7\}$.

To compare the above analysis with the periodic game of this paper, observe that the three oscillations used in the optimal mixed strategy above have periods $T_1=T_2=2$ and $T_3 = 12$, with least common multiple of $T^*=12$. So this also gives a solution to the periodic game with $n=7$ and $T=T^*=12$. Since $T^*$ is even and $n$ is odd our formula given in Proposition~\ref{prop:cases1-3}, case 2, is $2/(n+1) = 1/4$. The two analyses agree on the value. Note however, that the patrolling strategy given above differs from that given by our analysis of the periodic game with $T=12$ and $n=7$ given in Section~\ref{sec:general_results}, Figure \ref{fig:example_periodic}. Note also that for both patrolling strategies the nodes which are unfavourable to attack are the penultimate nodes $2$ and $6$. This shows that the Patroller strategies that we give in our analysis are not uniquely optimal. While this gives an alternative method of analyzing the periodic game $T=12$, $n=7$, it does not solve it in general. For example it would not solve the game for, say, $T=11$.

\section{Multiple Patrollers}
\label{sec:multiple}

We now consider a generalization of the game, where there are $k$ Patrollers. The Attacker's strategy set is the same, but his opponent chooses $k$ periodic walks on $L_n$, corresponding to $k$ patrols. The attack is intercepted and the payoff is $1$ if any of the Patrollers intercept the attack.

Let $V^{(k)}$ denote the value of the game when there are $k$ Patrollers,
and write $V^{(k)}_n$ for the value of the $k$ Patroller game on $L_n$.
Suppose in the single Patroller game the Patroller plays first as in the $k$
game but then picks a Patroller randomly. Thus he wins with probability at
least $V^{(k)}/k,$ and hence%
\begin{align}
V^{(k)} &\leq kV.  \label{eq:multiple}
\end{align}
That is, $k$ Patrollers can intercept an attack with probability at most
$k$ times the probability that a single Patroller can intercept an attack.

The estimate holds with equality if and only if the $k$ Patrollers can
jointly attack in such a way that each one is following an optimal strategy
for $k=1$ and furthermore no possible attack is simultaneously intercepted
by more than one of the Patrollers.

If we assume $k \le n/2$ then it is easy to adapt our optimal strategies described in the sections above for $k=1$ to the more general game where $k > 1$. As an example, take case 4, with $n=7$, $T=3$ and $k=3$. An optimal Patroller strategy for $k=1$ is depicted in Table~\ref{tab:L7-case4}: recall that the Patroller chooses one of the $12$ arrows at random and performs a left or right $p$-biased oscillation, depending on the direction of the arrow, where $p=6/7$.

An optimal strategy for $k=3$ simply chooses a row at random and assigns one of the $3$ Patrollers to each arrow. This clearly implies that $V_n^{(k)} = 3 V_n$. For $k=2$ the Patroller chooses a row at random and randomly assigns the $2$ Patrollers to $2$ of the $3$ arrows. Note that this extension to $k >1 $ Patrollers works for any $k \le 3$ but not for $k > 3$. For example this particular argument does not work for $k=4$. Note also that the alternative decomposed strategy for case 4, described in Section~\ref{sec:decomposed} cannot be extended to $k >1$ Patrollers in the same way.

Similarly, for the other cases, as long as $k \le n/2$, the Patroller's strategy for $k=1$ can be extended to $k >1$. We omit the details, as the extensions are straightforward. Hence we have the following theorem.

\begin{theorem}
	For $k \le n/2$, the value $V_n^{(k)}$ of the $k$ Patroller game on the line graph $L_n$ satisfies $V_n^{(k)} = k V_n$.
\end{theorem}

It is natural to question whether, for $k > n/2$, the value of the game is $\min\{k V_n,1\}$. Indeed, for $T$ even, it is easy to see that this is true, since for $k>n/2$, the Patroller can win the game with probability $1$ by oscillating on $k$ covering edges.

But for $T$ odd, it is not true. Consider the same example of $n=7$ and $T=3$ but this time with $k=4$ Patrollers. In
this case, the bound~(\ref{eq:multiple}) gives $V_{7}^{(4)}\leq 4V_{7}=20/21$.
Suppose the Attacker employs the uniform strategy. Since $T=3$ it is clear that each of
the $4$ Patrollers must either choose an edge and perform a biased
oscillation on that edge, or stay at a single node. Each Patroller can
only guarantee certain interception at only one node. It follows that there
are at most $4$ nodes at which any attack is intercepted with probability $1$%
, and the maximum probability an attack is intercepted at the remaining $3$
nodes is $2/3$. Hence the maximum probability of interception is $%
4/7 \cdot(1)+2/3\cdot(3/7)=6/7 < 20/21$, so $V_{7}^{(4)}\leq 6/7$ and~(\ref{eq:multiple}) is not
tight.

To see that the value $V_{7}^{(4)}$ is in fact exactly equal to $6/7$,
consider the strategy of the Patrollers as depicted in Table~\ref{tab:k=4}. This time
the circle with the dot in the middle indicates that a Patroller remains at
this node, whereas an arrow, as before, denotes performing a left or right $p$-biased oscillation, depending on the direction of the arrow, taking $p=4/7$. The Patrollers choose one of the four rows at random, then they are each assigned to one of the edges corresponding to an arrow or to the node corresponding to the circle with a dot in it.

All even numbered nodes have four arrows coming into them and using (\ref{eq:forward}) the
probability an attack there is intercepted is:
\begin{equation*}
p \cdot (1) + (1-p) \cdot \frac{T-1}{T} = (4/7) \cdot (1) + (3/7) \cdot (2/3) = 6/7.
\end{equation*}
Similarly, attacks at odd numbered nodes have a probability of $1/4$ that there is a stationary Patroller at that node who definitely intercepts the attack, and $3/4$ probability that there is a Patroller using an arrow going away from that node. Hence attacks at odd numbered nodes are intercepted with probability:
\begin{equation*}
1/4 (1) + 3/4 \left(\frac{T-p}{T} \right) = 6/7 \text{, by (\ref{eq:backward}).}
\end{equation*}

\begin{table}[htb!]
\begin{gather*}
\begin{tabular}{lllllllllllll}
1 &  & 2 &  & 3 &  & 4 &  & 5 &  & 6 &  & 7 \\
$\odot $ &  & $\bullet $ & $\Longleftarrow $ & $\bullet $ &  & $\bullet $ & $%
\Longleftarrow $ & $\bullet $ &  & $\bullet $ & $\Longleftarrow $ & $\bullet
$ \\
$\bullet $ & $\Longrightarrow $ & $\bullet $ &  & $\odot $ &  & $\bullet $ &
$\Longleftarrow $ & $\bullet $ &  & $\bullet $ & $\Longleftarrow $ & $\bullet $
\\
$\bullet $ & $\Longrightarrow $ & $\bullet $ &  & $\bullet $ & $%
\Longrightarrow $ & $\bullet $ &  & $\odot $ &  & $\bullet $ & $%
\Longleftarrow $ & $\bullet $ \\
$\bullet $ & $\Longrightarrow $ & $\bullet $ &  & $\bullet $ & $%
\Longrightarrow $ & $\bullet $ &  & $\bullet $ & $\Longrightarrow $ & $%
\bullet $ &  & $\odot $
\end{tabular}
\end{gather*}
	\caption{Optimal strategy for $L_{7}$ with $T=3$ and $k=4$.}
		\label{tab:k=4}
\end{table}

It is not hard to show that for $T=3$ and $n=7$, even for $k=6$ the value of the game is strictly less than $1$. In fact it is equal to $20/21$ in this case (we omit the details).

\section{Conclusions}
\label{sec:conclusions}

This paper has begun the study of periodic patrols on the line, by giving a
complete solution to the case of short attack duration $m=2$. One reason
that the case $m=2$ is susceptible to our analysis is that, at least for
even $T$, the covering number can be identified with the minimum number of
patrols that are required to intercept any attack. This is not true for
large $m$. The periodic patrolling game is much more difficult to solve than
the unrestricted version of the game (where patrols are not required to have
a given period). The latter can be solved for line graphs of arbitrary size
and arbitrary attack duration, as long as the time horizon is sufficiently
large, as shown in Papadaki et al. (2016).

\end{document}